\documentclass[12pt,reqno]{article}

\usepackage{amsthm,amsfonts,amssymb,amsmath,epsf, verbatim}

\newtheorem{theorem}{Theorem}
\newtheorem{lemma}[theorem]{Lemma}

\begin{document}

\title{An improvement of an inequality of Ochem and Rao concerning odd perfect numbers}

\author{Joshua Zelinsky}
\date{}

\maketitle
\vspace{-1.5 cm}

\begin{center}
Iowa State University\\

Email:joshuaz1@iastate.edu
\end{center}

\begin{abstract}

Let $\Omega(n)$ denote the total number of prime divisors of $n$ (counting multiplicity) and let $\omega(n)$ denote the number of distinct prime divisors of $n$. Various inequalities have been proved relating $\omega(N)$ and $\Omega(N)$  when $N$ is an odd perfect number.  We improve on these inequalities.  In particular, we show that if $3 \not| N$, then  $\Omega \geq \frac{8}{3}\omega(N)-\frac{7}{3}$   and if $3 |N$ then $\Omega(N) \geq \frac{21}{8}\omega(N)-\frac{39}{8}.$

\end{abstract}

\section{Introduction}
Let $\Omega(n)$ denote the total number of prime divisors of $n$ (counting multiplicity) and let $\omega(n)$ denote the number of distinct prime divisors of $n$.

Let $N$ be an odd perfect number. Ochem and Rao\cite{OchemRao1}  have proved that $N$ must satisfy \begin{equation}\Omega(N) \geq \frac{18\omega(N) -31}{7}\label{OR1}\end{equation} and \begin{equation}\Omega(N) \geq 2\omega(n) +51.\label{OR2} \end{equation} Note that Ochem and Rao's second inequality is stronger than the first as long as $\omega(N) \leq 81$. Nielsen has shown that $\omega(n) \geq  10 $. \cite{Nielsen}

In this note we improve Ochem and Rao's first inequality. In particular we have:

\begin{theorem} If $N$ is an odd perfect number, with $3\not| N$ then \begin{equation}\Omega(N) \geq \frac{8}{3}\omega(N)-\frac{7}{3}.\label{firstineq}\end{equation}   If $N$ is an odd perfect number, with $3 |N$ then \begin{equation}\Omega(N) \geq \frac{21}{8}\omega(N)-\frac{39}{8}.\label{secondineq}\end{equation}  
\end{theorem}

Inequality \ref{firstineq} is always better than inequality \ref{OR1}, while inequality \ref{secondineq}  becomes stronger than inequality \ref{OR1} when $\omega(N) \geq 9$ and thus for all odd perfect numbers by Nielsen's result.

Note that if one only uses Ochem and Rao's original system of inequalities but assumes that $3 \not| N$ then one can improve the constant term of inequality \ref{OR1} but one still has a linear coefficient of $18/7$. So, both cases here do represent non-trivial improvement. 

\section{Proof of the main results}

Our method of proof is very similar to that of Ochem and Rao; they created a series of linear inequalities involving the number of different types of prime factors (both total and distinct) of $N$ and showed that the linear system in question forced a certain lower bound. We will use a similar method, but with additional inequalities.

We we will write $\Omega(N)$ as just $\Omega$ and $\omega(N)$ as $\omega$.

Euler proved that $N$ must have the form $N=q^e m^2$ where $q$ is a prime such that $q \equiv e \equiv 1$ (mod 4), $(q,m)=1$. Traditionally $q$ is called the special prime.  Note that from this one one immediately has $\Omega(N) \geq 2\omega(N) - 1$. For the remainder of this paper we will assume that $N$ is an odd perfect with $q$, $e$ and $m$ given  as above.

The following Lemma is the primary insight that allows us to have a system that is tighter than that of Ochem and Rao: 

\begin{lemma}
If $a$ and $b$ are distinct odd primes and $p$ is a prime such that $p|(a^2+a+1)$ and $p|(b^2+b+1)$. If  $a \equiv b \equiv 2$ (mod 3), then $p \leq \frac{a+b+1}{5}$.  If $a \equiv b \equiv 1$ (mod 3) $p \leq \frac{a+b+1}{3}$.
\label{smallsharedprimes}
\end{lemma}
\begin{proof} We will prove this when $a \equiv b \equiv 2$ (mod 3) (the 1 mod 3 proof is nearly identical). 
Without loss of generality, assume that $a> b$. Note that one must have $p \equiv 1$ (mod 3). We have $$p|(a^2+a+1) - (b^2+b+1)= (a-b)(a+b+1).$$ So either $p|a-b$ or $p|a+b+1$. In the first case, we note that $6|a-b$, so $6p|a-b$ and  thus $$p \leq \frac{a-b}{6} \leq \frac{a+b+1}{5}.$$ In the second case, we have that $p|a+b+1$ gives us $pk=a+b+1$ for some $k$ with $k \equiv 5$ (mod 6) and so $k \geq  5$. Thus, $$p \leq \frac{a+b+1}{5}.$$ Thus in both cases we have the desired inequality. 

\end{proof}

Note that we do not have a version of Lemma \ref{smallsharedprimes}  when $a \equiv 1$ (mod 3) and $b \equiv 2$ (mod 3), since one cannot in that case get beyond $p\leq a+b+1$, and the case of $a=7$, $b=11$ and $p=19$ shows that one can in fact have $p=a+b+1$ and for the method we will use our lemma we need an inequality strong enough that we can conclude that $p < \max(a,b)$.

We will also need the following result, which is Lemma 3 in Ochem and Rao:
\begin{lemma} Let $p$, $q$ and $r$ be positive integers. If $p^2+p+1=r$ and $q^2+q+1=3r$ then $p$ is not an odd prime. 
 
\end{lemma}

Now, for the proof of the main result:\\

We will write $$S= \prod_{p||m, p \neq 3} p$$ and $$T =  \prod_{p^2 | m, p \neq 3} p.$$ We will set $S=S_1S_2S_3$ where a prime $p$ appears in $S_i$ for $1 \leq i \leq 2$ if $\sigma(p^2)$ is a product of $i$ primes; $S_3$ will contain all the primes of $S$ where $\sigma(p^2)$ has at least 3 prime factors. We will write  $s= \omega(S)$ and write $t= \omega(T)$. We define $s_1$, $s_2$ and $s_3$ similarly. 

We will write $S_{i,j}$ to be the primes from $S_i$ which are $j$ (mod 3).  In a similar way to  use lower case letters to denote the number of primes in each term as before. That is, we set $s_{i,j}=\omega(S_{i,j})$ and will note that $s_{1,1}=0$. Thus, we do not need to concern ourselves with this split for $S_1$  since all primes in $S_1$ are 2 (mod 3) there is no need to split $S_1$ further.

We have the special exponent is at least $1$:  \begin{equation} 1 \leq e \label{specialexists} \end{equation}

We have the following straightforward equations from breaking down the definitions of $s_1,s_2,$ and $s_3$:

\begin{equation}
    s=s_1 +s_2 + s_3. \label{sbreakdown}
\end{equation}

Similarly, we have:  \begin{equation}
s_2=s_{2,1} + s_{2,2}, \label{s2breakdown}
\end{equation} and we have \begin{equation}
s_3=s_{3,1} + s_{3,2} \label{s3breakdown}
\end{equation}

We  define $f_4$ as the number of prime divisors (counting multiplicity) in $N$ which are not the special prime and are raised to at least the fourth power. From simple counting we obtain: 

\begin{equation} e+f_3+2s +f_4 \leq \Omega.  \label{bigomegalower2}
\end{equation} 

\begin{lemma} We have \begin{equation}
 s_1 + s_{2,2}  \leq t + s_{2,1}+ s_{3,1} +1 \label{s1s22upper}\end{equation} 

and \begin{equation} s_1   \leq t +  s_{3,1}+1.
\label{s1s21upper}
\end{equation}
\label{ts1s2}
\end{lemma}
\begin{proof}

The claim will be proven if we can show that each $p$ in $S_1S_{2,i}$ has to contribute at least one distinct prime (since then one of the primes may be the special prime and the other primes must all contribute to $t$. This is trivial for $S_1$ (and was used in Ochem and Rao's result). We will show this by showing that no two prime divisors of $S_1S_{2,i}$ can contribute the same largest prime of the primes they contribute. We have two cases we need to consider, both primes arising from $S_{2,i}$ or one arising from $S_{2,i}$ and one arising from $S_1$.\\

Case I: Assume we have two prime divisors of $S_{2,2}$, $a$ and $b$ with $a>b $ and assume they have some shared prime factor $q$ which divides both $\sigma(a^2)=a^2 + a +1$ and $\sigma(b^2)=b^2 +b + 1$. Since we have $a \equiv b \equiv $ 2 (mod 3), we may apply Lemma \ref{smallsharedprimes} to conclude that that $$q \leq \frac{a+b+1}{5} \leq \frac{3a}{5}.$$ Thus, $$\frac{\sigma(a^2)}{q} > \frac{a^2}{3a/5} = 5/3a > (a^2+a+1)^{1/2}.$$ But $\sigma(a^2)$ only has two prime factors, and this shows that the shared contributed prime cannot be the largest prime contributed by $a$.  The case of two primes dividing $S_{2,1}$ is similar. \\

Case II: Assume that we have a prime $a$ dividing $S_1$ and a prime $b$ dividing $S_{2,i}$. If $i=2$ then the same logic as above works. So assume that $i=1$, and $b$ $\equiv$ (1 mod 3). Thus we have a prime $p$ such that $a^2+a+1=p$ and $b^2+b+1=3p$ for some prime $p$. This is precisely the situation ruled out by Ochem and Rao's Lemma. 

\end{proof}

Next we have \begin{equation}s_{2,1}+s_{3,1} \leq f_3, \label{3lower} \end{equation} since if $x \equiv 1$ (mod 3), then $x^2+x+1 \equiv 0$ (mod 3).

\begin{comment}
This next bit follows Ochem and Rao closely. Let $K$ be the multiset of primes distinct from $3$ produced by the $S$ contribution. The primes in $K$ are all $1$ (mod 3) (since the only non-three primes which can divide a number of the form $x^2+x+1$ are 1 (mod 3)). Thus, if we define $f_4$ to be the total number of prime factors (counting multiplicity) which are primes of exponent of at least four and are not the special prime or three, then  $$|K| \leq |K| <= e+f_4+2s_{2,1}+2s_{3,1}.$$ 

For prime $p>3$, define $\alpha(p)= \sigma(p^2)$ when $u \equiv 2$ (mod 3), and $\alpha(p)=\sigma(p^2)/3$ when $u \equiv 1$ (mod 3). Thus by Lemma \ref{ORunique}, if $p_1$ and $p_2$ are primes, then $\alpha(p_1)=\alpha(p_2)$ if and only if $p_1=p_2$. Thus, any prime in $S$ must contribute at least 2 primes to $K$ (not necessarily distinct), less one distinct prime for each element of $K$. Thus, we have $$2s -1 -s_{2,1}-s_{3,1} -t \leq |K|.$$  

Combining these two inequalities we have $$2s -1 -s_{2,1}-s_{3,1} -t \leq |K| \leq e+f_4+2s_{2,1}+2s_{3,1}$$ which simplifies to \begin{equation}2s  \leq e+t+f_4+3s_{2,1}+3s_{3,1}+1. \label{OchemandRaomodified} \end{equation}
\end{comment}

We also have by counting all the $1$ (mod $3$) primes which are contributed by primes in $S$: $$s_1+2s_{2,2}+3s_{3,2} + s_{2,1} + 2s_{3,1} \leq f_4+e + 2s_{2,1}+2s_{3,1}. $$ This simplifies to:

\begin{equation}s_1+2s_{2,2}+3s_{3,2}   \leq f_4+e +s_{2,1}. \label{sbruteupper} \end{equation}

And we of course have \begin{equation}4t \leq f_4 \label{4tf4}.\end{equation}

If $3 \not| N$ then our system of equations and inequalities also includes the additional constraint $f_3 =0$ as well \begin{equation}
    \omega= s+ t + 1.  \label{omegabreakdownwithout3}
\end{equation}

Since $f_3=0$ we have  $s_{2,1}=s_{3,1}=0$ and after zeroing those variables we obtain inequality \ref{firstineq} by taking $7/9 \times  {\bf{(\ref{specialexists})}} +  2/3 \times {\bf{(\ref{sbreakdown})}} + 2/3 \times {\bf{(\ref{s2breakdown})}} + 2/3 \times {\bf{(\ref{s3breakdown})}} + 1 \times {\bf{(\ref{bigomegalower2})}} + 4/9 \times {\bf{(\ref{s1s22upper})}} +2/9 \times {\bf{(\ref{sbruteupper})}} +7/9 \times {\bf{(\ref{4tf4})}} + 8/3 \times {\bf{(\ref{omegabreakdownwithout3})}}   $ where the bold numbers represent the corresponding numbered equation or inequality.\\

If $3|N$ we have $2 \leq f_3$ and  
\begin{equation}
    \omega= s+ t + 2, \label{omegabreakdownwith3} 
\end{equation}    
    
    Similar to the previous case, the linear combination $3/4 \times  {\bf{(\ref{specialexists})}} +5/8\times {\bf{(\ref{sbreakdown})}} + 5/8 \times {\bf{(\ref{s2breakdown})}} + 5/8 \times {\bf{(\ref{s3breakdown})}} + 1 \times {\bf{(\ref{bigomegalower2})}} + 1/8 \times {\bf{(\ref{s1s22upper})}} + 1/4 \times {\bf{(\ref{s1s21upper})}}+1\times {\bf{(\ref{3lower})}} +1/4\times  {\bf{(\ref{sbruteupper})}} +3/4 \times {\bf{(\ref{4tf4})}} +21/8 \times {\bf{(\ref{omegabreakdownwith3})}}$ yields inequality \ref{secondineq}.\\
    
{\bf Acknowledgments}     Pascal Ochem greatly assisted in early drafts of this paper both in exposition and in clarifying the results. Maria Stadnik also made helpful suggestions.

\end{document}